\documentclass[12pt,reqno]{amsart}
\allowdisplaybreaks

\usepackage{amstext}
\usepackage{amsthm}
\usepackage{amsmath}
\usepackage{amssymb}
\usepackage{latexsym}
\usepackage{amsfonts}
\usepackage{graphicx}
\usepackage{texdraw}
\usepackage{graphpap}
\usepackage{color}
\usepackage{bbm}
\usepackage[pagebackref,hypertexnames=false, colorlinks, citecolor=red, linkcolor=red]{hyperref} %

\usepackage[backrefs]{amsrefs}

\input txdtools

\begin{document}

\newcommand{\ci}[1]{_{ {}_{\scriptstyle #1}}}

\newcommand{\norm}[1]{\ensuremath{\left\|#1\right\|}}
\newcommand{\abs}[1]{\ensuremath{\left\vert#1\right\vert}}
\newcommand{\ip}[2]{\ensuremath{\left\langle#1,#2\right\rangle}}
\newcommand{\p}{\ensuremath{\partial}}
\newcommand{\pr}{\mathcal{P}}

\newcommand{\pbar}{\ensuremath{\bar{\partial}}}
\newcommand{\db}{\overline\partial}
\newcommand{\D}{\mathbb{D}}
\newcommand{\B}{\mathbb{B}}
\newcommand{\Sp}{\mathbb{S}}
\newcommand{\T}{\mathbb{T}}
\newcommand{\R}{\mathbb{R}}
\newcommand{\Z}{\mathbb{Z}}
\newcommand{\C}{\mathbb{C}}
\newcommand{\N}{\mathbb{N}}
\newcommand{\mQ}{\mathcal{Q}}
\newcommand{\mS}{\mathcal{S}}
\newcommand{\scrH}{\mathcal{H}}
\newcommand{\scrL}{\mathcal{L}}
\newcommand{\td}{\widetilde\Delta}
\newcommand{\pw}{\text{PW}}
\newcommand{\esup}{\text{ess.sup}}
\newcommand{\Tn}{\mathcal{T}_n}
\newcommand{\Bn}{\mathbb{B}_n}
\newcommand{\rt}{\mathcal{O}}
\newcommand{\avg}[1]{\langle #1 \rangle}
\newcommand{\one}{\mathbbm{1}}
\newcommand{\eps}{\varepsilon}

\newcommand{\La}{\langle }
\newcommand{\Ra}{\rangle }
\newcommand{\rk}{\operatorname{rk}}
\newcommand{\card}{\operatorname{card}}
\newcommand{\ran}{\operatorname{Ran}}
\newcommand{\osc}{\operatorname{OSC}}
\newcommand{\im}{\operatorname{Im}}
\newcommand{\re}{\operatorname{Re}}
\newcommand{\tr}{\operatorname{tr}}
\newcommand{\vf}{\varphi}
\newcommand{\f}[2]{\ensuremath{\frac{#1}{#2}}}

\newcommand{\kzp}{k_z^{(p,\alpha)}}
\newcommand{\klp}{k_{\lambda_i}^{(p,\alpha)}}
\newcommand{\TTp}{\mathcal{T}_p}
\newcommand{\m}[1]{\mathcal{#1}}
\newcommand{\md}{\mathcal{D}}
\newcommand{\qan}{\abs{Q}^{\alpha/n}}
\newcommand{\sbump}[2]{[[ #1,#2 ]]}
\newcommand{\mbump}[2]{\lceil #1,#2 \rceil}
\newcommand{\cbump}[2]{\lfloor #1,#2 \rfloor}


\newcommand{\entrylabel}[1]{\mbox{#1}\hfill}

\newenvironment{entry}
{\begin{list}{X}%
  {\renewcommand{\makelabel}{\entrylabel}%
      \setlength{\labelwidth}{55pt}%
      \setlength{\leftmargin}{\labelwidth}
      \addtolength{\leftmargin}{\labelsep}%
   }%
}%
{\end{list}}


\numberwithin{equation}{section}

\newtheorem{thm}{Theorem}[section]
\newtheorem{lm}[thm]{Lemma}
\newtheorem{cor}[thm]{Corollary}
\newtheorem{conj}[thm]{Conjecture}
\newtheorem{prob}[thm]{Problem}
\newtheorem{prop}[thm]{Proposition}
\newtheorem*{prop*}{Proposition}
\newtheorem{claim}[thm]{Claim}

\theoremstyle{remark}
\newtheorem{rem}[thm]{Remark}
\newtheorem*{rem*}{Remark}
\newtheorem{defn}[thm]{Definition}

\hyphenation{geo-me-tric}

\title[Entropy Bounds for Fractional Operators]
{Some Entropy Bump Conditions for Fractional Maximal and Integral Operators}

\author[R. Rahm]{Robert Rahm}
\address{Robert Rahm, School of Mathematics\\ Georgia Institute of Technology\\ 686 Cherry 
    Street\\ Atlanta, GA USA 30332-0160}
\email{rrahm3@math.gatech.edu}
\author[S. Spencer]{Scott Spencer}
\address{Scott Spencer, School of Mathematics\\ Georgia Institute of Technology\\ 686 Cherry 
    Street\\ Atlanta, GA USA 30332-0160}
\email{spencer@math.gatech.edu}


\begin{abstract}
We investigate weighted inequalities for fractional maximal operators and 
fractional integral operators. We work within the innovative framework of 
``entropy bounds'' introduced by Treil--Volberg. Using techniques developed
by Lacey and the second author, we are able to efficiently prove 
the weighted inequalities. 
\end{abstract}

\maketitle
\section{Introduction}
We are concerned with two-weight inequalities for the fractional maximal and fractional
integral operators. The goal is to find simple, $A_p-$like conditions for a pair of weights 
(non--negative, locally integrable functions) $\sigma,w$ to ensure 
\begin{align} 
\|T^\sigma:L^p(\sigma)\to L^q(w)\|<\infty, \label{goal} 
\end{align} 
where $T$ denotes a fractional maximal or fractional integral operator, 
and $T^\sigma(f):=T(\sigma f)$.

One popular approach, initiated by Neugebauer in \cite{Neu} and developed by P\'erez in \cites{Per, Per1}, has been to 
slightly strengthen the $A_p$ characteristic by introducing new factors.
These new factors, known as bumps, have come in different forms. For
example, Neugebauer requires that the weights 
$\sigma^{1+\epsilon}$ and $w^{1+\epsilon}$ belong to $A_p$, while P\'{e}rez 
requires that the two weights have finite Orlicz norm.  
The Orlicz approach is also taken by Cruz-Uribe and Moen in \cite{CruMoe1}.  See the recent paper of Cruz--Uribe 
\cite{Cru} and the references therein for more information.

In the context of Calder\'{o}n--Zygmund operators, 
Treil--Volberg have recently introduced the 
notion of \textit{entropy bounds} and are able to deduce 
stronger results than have been obtained using the 
Orlicz approach \cite{TreVol}. In \cite{LacSpe}, Lacey and the second
author combine the entropy bound approach with the
theory of sparse operators, introduced by Lerner \cite{Ler},
to efficiently deduce stronger results
than in \cite{TreVol}. We use these techniques to prove similar 
results for fractional integral and fractional maximal operators. 
In particular, we require that our weights satisfy certain
``bump" or ``separated bump'' conditions (to be defined 
below.) 

Before stating the main theorems, we give some definitions. 
For $0<\alpha<n$, the fractional maximal operator for
functions defined on $\mathbb{R}^n$ is
\begin{align*}
M_{\alpha}f(x) 
:= \sup_{Q\text{ a cube}}\frac{\one_{Q}(x)}{\abs{Q}^{1-\frac{\alpha}{n}}}
\int_{Q}\abs{f(y)}dy,  
\end{align*}
and the fractional integral operator is
\begin{align*}
I_{\alpha}f(x) :=
\int_{\mathbb{R}^n}\frac{f(y)}{\abs{y-x}^{n-\alpha}}dy.
\end{align*}

One reasonable generalisation of the Muckenhoupt $A_p$ condition
to the present setting is to set $[\sigma,w]:=\sup_{Q\text{ a cube}}\sigma(Q)^{1/p'}w(Q)^{1/q}\abs{Q}^{\alpha/n-1}$. 
Ideally, we would like for \eqref{goal} to hold when 
$[\sigma,w]$ is finite. This condition is insufficient (see 
\cite{CruMoe} for a counter example in the case
of the fractional maximal operator). This condition \textit{is} enough, however, to deduce \textit{weak-type} bounds for the maximal operator
(we present an alternate proof of this well--known result in Section \ref{weakmax}); in particular, there holds:
\begin{thm}\label{weakmax}
With $[\sigma,w]$ defined as above, $M_\alpha$ the fractional 
maximal operator, and $1\leq p \leq q \leq \infty$, there holds: 
\begin{align*}
\norm{M_\alpha(\sigma\cdot):L^p(\sigma)\to L^{q,\infty}(w)}\lesssim [\sigma,w].
\end{align*}
\end{thm}  

\begin{rem}
In an earlier draft of this paper, we claimed that the above inequality holds
for the fractional integral operator as well. This is incorrect and we would
like to thank Kabe Moen for pointing out this error. 
\end{rem}

Since the finiteness of $[\sigma,w]$ is not enough to deduce
strong bounds, we use two types of bumped conditions to 
deduce the strong estimates. 
The first set of conditions on the weights that we consider 
require a single bump (compare with the 
separated bumps to be discussed later). Set $ \rho_\sigma(Q) := \frac{1}{\sigma(Q)}\int_{Q}M(\sigma \one_Q)$,
and define $\rho_w$ similarly, where $M$ is the 
Hardy--Littlewood maximal operator. We deal first with the fractional maximal
operator.

\begin{thm}\label{maxbump}
Let $\sigma$ and $w$ be two weights, $1<p\leq q < \infty$, and $M_\alpha$ be the
fractional maximal operator. 
Let $\epsilon_q$ be a monotonic increasing function on $(1,\infty)$ that satisfies 
$\int_{1}^{\infty}\frac{dt}{t\epsilon_q^q(t)}=1$. 
Define
\begin{align*}
\beta(Q)
:=\frac{\sigma(Q)^{1/p'}w(Q)^{1/q}}{\abs{Q}^{1-\alpha/n}}\rho_{\sigma}^{1/p}(Q)\epsilon_q(\rho_{\sigma}(Q)),
\end{align*}
and set $\mbump{\sigma}{w} := \sup_{Q\in\mathcal{D}}\beta(Q)$.  Then
\begin{align*}
\norm{M_{\alpha}(f\sigma)}_{L^q(w)}
    \lesssim \mbump{\sigma}{w} \norm{f}_{L^p(\sigma)}.
\end{align*}
\end{thm}

The corresponding theorem for the fractional integral operator is:
\begin{thm}\label{frac1bump}
Let $1\leq p \leq \infty$ and $\sigma$ and $w$ be two weights and let $I_\alpha$ be the
fractional integral operator. Let $\epsilon_p$ be a monotonic increasing function on
$(1,\infty)$ that satisfies $\int_{1}^{\infty}\frac{dt}{t\epsilon_p^p(t)}=1$ 
and similarly for $\epsilon_{q'}$. Define:
\begin{align*}
\beta(Q)
:=\frac{\sigma(Q)^{1/p'}w(Q)^{1/q}}{\abs{Q}^{1-\alpha/n}}
    \rho_\sigma(Q)^{1/p}\epsilon_p(\rho_{\sigma}(Q))
    \rho_{w}(Q)^{1/q'}\epsilon_{q'}(\rho_w(Q)),
\end{align*}
and set $\cbump{\sigma}{w} := \sup_{Q\in\mathcal{Q}}\beta(Q)$.  Then
\begin{align*}
\norm{I_\alpha (f\sigma)}_{L^q(w)}
    \lesssim C_{\alpha,n} \cbump{\sigma}{w} \norm{f}_{L^p(\sigma)}.
\end{align*}
\end{thm}

The $C_{\alpha,n}$ constant in the above and below theorems arise below in $\eqref{can}$.  

The condition in the next theorem is called a ``separated bump'' for obvious reasons. 
We use a bump defined in terms of the fractional maximal operator, namely 
\[\varrho_\sigma^{\alpha,p,q}(Q):= \dfrac{\int_Q M_\alpha(\one_Q \sigma)^{q/p} dx}{\sigma(Q)^{q/p}},\]
or simply $\varrho_\sigma$ or $\varrho$ when clear.  We have the following

\begin{thm}\label{sbump}  Let $\sigma$ and $w$ be weights with densities, $1<p\leq q<\infty$, and $\eps_q,\eps_{p'}:\R^+\to\R$ be nonincreasing on $(0,1)$ and nondecreasing on $(1,\infty)$ such that $\int_0^\infty \frac{dt}{t\eps^{1/q}_q(t)},\int_0^\infty \frac{dt}{t\eps^{1/p'}_{p'}(t)} <\infty$.  Define 

\begin{align*} \sbump{\sigma}{w}_{\alpha,p,q}:=\sup_{Q \text{ a cube}}  \left(|Q|^{\alpha/n}\avg{\sigma}_Q\right)^{q/p'} \avg{w}_Q \varrho_\sigma^{\alpha,p,q}(Q)\eps_q\left(\varrho_\sigma^{\alpha,p,q}(Q)\right).
\end{align*}  There holds:

\begin{align*}
\|I_\alpha^\sigma:L^p(\sigma)\to L^q(w)\|\lesssim C_{\alpha,n} \left( \sbump{\sigma}{w}_{\alpha,p,q}^{1/q} + \sbump{w}{\sigma}_{\alpha,q',p'}^{1/p'} \right).
\end{align*}
\end{thm}

In Section 2, we give some preliminary information and lemmas that will be used below. 
In Section 3, we give a proof of the weak estimates. Sections 4 and 5 contain the proofs of
the one--bump theorems for the fractional maximal and fractional integral operators. The 
proofs in these sections use the theory of sparse operators, discussed below, but avoid the
explicit use of testing inequalities. Finally, Section 6 contains the proof of the 
separated bump theorem for the fractional integral operator. The proof uses both 
sparse operators and testing inequalities but is still elementary.

\section{Preliminaries}
In this section, we list several well--known results; we include some proofs because
we could not find them in the literature.  We start with some familiar definitions. 
For a measure $\mu$,  will write $\avg{f}_Q^\mu$ for $\frac{1}{\mu(Q)}\int_{Q}f$ and 
$\avg{f}_Q$ when $\mu$ is Lebesgue measure. 
\begin{defn}
A collection, $\md$ of cubes is said to be a \textit{dyadic grid} if:
\begin{itemize}
\item [(i)] The side length of every $Q\in\md$ equals $2^k$ for some
    $k\in\mathbb{Z}$.
\item [(ii)] If $Q,R\in\md$ and $Q\cap R$ is not empty then either
    $Q\subset R$ or $R\subset Q$.
\item [(iii)] If $\m{D}_k=\{Q\in\m{D}:\text{the side length of Q equals }2^k\}$, then
    $\mathbb{R}^n=\cup_{Q\in\m{D}_k}Q$.
\end{itemize}
\end{defn}

\begin{defn}
A subset $\m{S}$ of a dyadic grid is said to be sparse if
for every $P\in\m{S}$ there holds:
\begin{align*}
\sum\limits_{\substack{Q\in\md:Q\subsetneq P \\ Q\text{ is maximal}}}\abs{Q}\leq \frac{1}{2}\abs{P}.
\end{align*}
\end{defn}

\begin{defn}
Given a measure $\mu$ on $\mathbb{R}^n$ and a dyadic grid, $\md$, a sequence of positive numbers, 
$\{a_Q\}_{Q\in\mathcal{D}}$, is called a
$p,q$--\textit{Carleson Sequence} if for every $P\in\md$,
\begin{align}\label{cardef}
\frac{1}{\mu(P)^{q/p}}\sum_{Q\in\md:Q\subset P}a_Q \lesssim 1.
\end{align}
\end{defn}

\begin{lm}\label{cet}
Let $\mu$ be a measure on $\mathbb{R}^n$, $\md$ be a dyadic grid, and $\{a_Q\}_{Q\in\md}$ be
a $p,q$--Carleson Sequence. If $1 < p \leq q < \infty$, there holds:
\begin{align*}
\sum_{Q\in\mathcal{D}} a_Q\left(\avg{f}_Q^\mu\right)^{q}
\lesssim \norm{f}_{L^p(\mu)}^q,
\end{align*}
where the implied constant depends on $p,q$ and the best constant in \eqref{cardef}.
\end{lm}
\begin{proof}
We will treat $\mathcal{D}$ as a discrete measure space with measure $\nu$ where
$\nu(Q)=a_Q$. We show that the operator $T$ with rule $(Tf)(Q)=\avg{f}_Q^\mu$ satisfies $\norm{Tf}_{L^q(\nu)}^q\lesssim \norm{f}_{L^p(\mu)}^q$. The objective then
is to show that for every $\lambda > 0$, there holds:
\begin{align}\label{cetwww}
\lambda^q\nu(\{Tf>\lambda\}) \lesssim \left(\lambda^p\mu(Mf>\lambda)\right)^{q/p},
\end{align}
where $M$ is the dyadic maximal function. The lemma follows from \eqref{cetwww}
since the dyadic maximal function is bounded for $p>1$:
\begin{align*}
\norm{Tf}_{L^q(\nu)}^q
\simeq\sum_{k\in\mathbb{Z}}2^{kq}\nu(\{Tf>2^k\})
\lesssim \left(\sum_{k\in\mathbb{Z}}2^{kp}\mu(\{Mf>2^k\})\right)^{q/p}
\simeq \norm{Mf}_{L^p(\mu)}^{q/p}. 
\end{align*}
We now turn to proving 
\eqref{cetwww}. Fix $\lambda > 0$, and let 
$\mathcal{D}_\lambda$ be the maximal elements $Q\in \mathcal{D}$ such
that $\avg{f}_Q^{\mu} > \lambda$ (such maximal cubes exist since $f\in L^p(\mu)$). Using
the Carleson property of the sequence $\{a_Q\}_{Q\in\md}$, there holds:
\begin{align*}
\lambda^q\nu(\{Tf>\lambda\}) 
=\lambda^q\sum_{P\in\md_\lambda}\sum_{Q\in\md_\lambda:Q\subset P}a_Q
\leq \sum_{P\in\md_\lambda} (\lambda^p\mu(P))^{q/p}
\leq (\lambda^p \mu(\{Mf>\lambda\}))^{q/p}.
\end{align*}
The last inequality follows by the disjointness of the $P\in\md_\lambda$ and
the fact that $q/p\geq 1$.
\end{proof}
For the ``continuous'' version of this theorem, see \cite{Dur}. We are certain that 
Lemma \ref{cet} is contained in a paper, but we have not been able to find a
reference. 

For a given dyadic grid, $\mathcal{D}$, define the dyadic fractional maximal 
operator:
\begin{align*}
 M_{\alpha}^{\m{D}}f(x):=\sup_{Q\in\m{D}}\one_Q(x)\abs{Q}^{\alpha/n}\avg{f}_Q
\end{align*}
and the dyadic fractional integral operator:
\begin{align*}
 I_{\alpha}^{\m{D}}f(x):=\sum_{Q\in\m{D}}\abs{Q}^{\alpha/n}\avg{f}_Q\one_Q(x).
\end{align*}
The following lemma is well--known (for the proof of the fractional integral 
estimate see \cite{CruMoe1}; the proof of the estimate for the 
fractional maximal operator is obvious given the fact that for 
every cube, $Q$, there is a cube, $P_Q$ in a dyadic grid such that
$Q\subset P_Q$ and $\abs{P_Q}\leq 3^n \abs{Q}$):
\begin{lm}\label{mf2d}\label{ff2d}
Let $M_{\alpha}$ be the fractional maximal operator and $I_{\alpha}$ be the 
fractional integral operator. There is a collection of $3^n$
dyadic grids such that the following point--wise equivalences hold
for all non--negative $f$:
\begin{align*}
M_{\alpha}f\simeq\sum_{k=1}^{3^n}M_{\alpha}^{\m{D}_k}f
\hspace{.2in}
\text{ and }
\hspace{.2in}
I_{\alpha}f\simeq\sum_{k=1}^{3^n}I_{\alpha}^{\m{D}_k}f.
\end{align*}
\end{lm}

\begin{rem}\label{reduction}
When proving the estimates below for the dyadic fractional maximal operator, it is more 
convenient to deal with the following truncated version:
\begin{align}\label{reduction1}
\one_{Q_0}(x)\sup_{Q\in\md:Q\subset Q_0}\qan \avg{f}_Q\one_Q(x).
\end{align}
We then prove estimates that are independent of $Q_0$ and appeal to the monotone
convergence theorem to conclude the desired results. Assuming that $f$ is finite 
almost everywhere (which will always be the case for us), we can further 
simplify matters. We start by building a stopping collection, $\m{S}$. Initialise
$\{Q_0\}\to\m{S}$, and in the recursive stage, if $P\in\m{S}$ is minimal, add to 
$\m{S}$ all maximal children $Q$ of $P$ such that $\qan\avg{f}_Q>4\abs{P}^{\alpha/n}\avg{f}_P$.
For a cube $Q\subset Q_0$, let $Q^S$ denote the $\m{S}$--parent of $Q$. Similarly, 
let $\text{ch}(S)$ denote the maximal $\m{S}$--descendants of $S$. Finally, 
let $E_Q=Q\setminus\text{ch}(Q)$. A simple computation shows that for every $S\in\m{S}$, 
\begin{align*}
\sum_{Q\in\text{ch}(S)}\abs{S} \leq \frac{1}{2}\abs{S}
\hspace{.25in}
\text{ and }
\hspace{.25in}
\abs{S}\leq 2 \abs{E_S}. 
\end{align*}
That is, the stopping collection $\m{S}$ is sparse. Additionally, the $E_Q$ are
pairwise disjoint and for almost every $x\in Q_0$ there is some $Q$ with $x\in E_Q$
(this follows from the fact that $f=\infty$ on a set of measure zero).
Thus, we may further reduce \eqref{reduction1} to:
\begin{align}\label{reduction2}
\one_{Q_0}(x)\sup_{Q\in\md:Q\subset Q_0}\qan \avg{f}_Q\one_Q(x)
=\sum_{Q\in\m{S}:Q\subset Q_0}\qan \avg{f}_Q\one_{E_Q}(x).
\end{align}
We also note that if $\{E_Q\}_{Q\in\m{D}}$ is any collection of pairwise
disjoint sets such that $E_Q\subset Q$, then 
$\sum_{Q\in\m{D}}\abs{Q}^{\alpha/n}\avg{f}_Q\one_{E_Q}(x) \leq M_\alpha f(x)$.  

There is a similar reduction for the dyadic fractional integral
operator. Again, we may reduce matters to:
\begin{align}\label{reduction3}
\one_{Q_0}(x)\sum_{Q\in\md:Q\subset Q_0}\qan \avg{f}_Q\one_Q(x).
\end{align}
We now create the stopping family by initialising $\{Q_0\}\to\m{S}$ and 
in the recursive stage, if $P\in\m{S}$ is minimal, add to 
$\m{S}$ all maximal children $Q$ of $P$ such that 
$\avg{f}_Q>4\avg{f}_P$. Note that we are stopping on \textit{averages},
not fractional averages.  Again, simple computations show that $\m{S}$ is sparse. For
fixed $x\in Q_0$, and fixed $S\in\m{S}$, the sequence $\{\qan \one_Q(x)\}_{Q^S=S}$
is geometric and so 
\begin{align}
\sum_{Q^S=S}\qan \one_Q(x)\simeq C_{\alpha,n} \abs{S}^{\alpha/n}\one_S(x).\label{can}
\end{align}
Therefore, the sum in \eqref{reduction3} can be estimated as:
\begin{align}\label{reduction4}
\sum_{S\in\m{S}}\sum_{Q^S=S}\qan \avg{f}_Q\one_Q(x)
\lesssim\sum_{S\in\m{S}}\avg{f}_S\sum_{Q^S=S}\qan \one_Q(x)
\lesssim \sum_{S\in\m{S}}\abs{S}^{\alpha/n}\avg{f}_S\one_S(x).
\end{align}

Therefore, in all estimates below, for fixed $f$, we can replace the operator
of interest with one from the right hand
side of \eqref{reduction2} or \eqref{reduction4}; our 
estimates will be independent of sparse collection $\mS$ and root $Q_0$.
\qed
\end{rem}

We have the following well--known theorem, originally due to 
Sawyer. See \cites{Saw1, Hyt, LacSawUri}.
\begin{lm}\label{testingdyad}
Let $1<p\leq q < \infty$, let $\md$ be a dyadic grid and let $\m{S}\subset \md$ be sparse. Let $T$ be the operator
given by $Tf=\sum_{Q\in\m{S}}\abs{Q}^{\alpha/n}\avg{f}_Q\one_Q$. Define:
\begin{align*}
\beta_1 :=
\sup_{P\in\m{S}}\frac{1}{\sigma(P)^{q/p}}\int_{P}
    \abs{\sum_{Q\in\m{S}:Q\subset P}\abs{Q}^{\alpha/n}\avg{\sigma}_Q\one_Q(x)}^{q}w(x)dx,
\end{align*}
\begin{align*}
\beta_2 :=
\sup_{P\in\m{S}}\frac{1}{w(P)^{p'/q'}}\int_{P}
    \abs{\sum_{Q\in\m{S}:Q\subset P}\abs{Q}^{\alpha/n}\avg{w}_Q\one_Q(x)}^{p'}\sigma(x)dx.
\end{align*}
Then:
\begin{align*}
\norm{T_\sigma:L^p(\sigma)\to L^q(w)}\lesssim \beta_1+\beta_2.
\end{align*}
\end{lm}

\section{Proof of Theorem \ref{weakmax}}
By Lemma \ref{mf2d}, 
Theorem \ref{weakmax} follows from the following lemma.

\begin{lm}
Let $1\leq p \leq q < \infty$ and $\sigma$ and $w$ be two weights. Let $\md$ be a dyadic grid, 
and let $M_\alpha$ the dyadic fractional integral operator. Define:
\begin{align*}
\beta(Q)=\frac{\sigma(Q)^{1/p'}w(Q)^{1/q}\qan}{\abs{Q}}.
\end{align*}
Set $[\sigma,w]:=\sup_{Q\in\mathcal{D}}\beta(Q)$, then
\begin{align}\label{weakest1}
\lambda^q w(\{I_\alpha f>\lambda\}) \lesssim [\sigma,w]^q\norm{f}_{L^p(\sigma)}^q. 
\end{align}
\end{lm}
\begin{proof}
Let $\md_\lambda$ be the maximal elements of $\md$ contained in $Q_0$ such that
$\qan\avg{f\sigma}_Q > \lambda$. Since $\avg{f\sigma}_Q = 
\avg{f}_Q^\sigma\avg{\sigma}_Q$, there holds:
\begin{align*}
\lambda^q w\{Mf>\lambda\} 
  \leq \sum_{Q\in\md_\lambda}\lambda^q w(Q)
  \leq \sum_{Q\in\md_\lambda}\abs{Q}^{\frac{q\alpha}{n}}\avg{\sigma}_Q^q w(Q) 
    \left(\avg{f}_Q^\sigma\right)^q
  \leq [\sigma,w]^q \sum_{Q\in\md_\lambda}
  \sigma(Q)^{\frac{q}{p}}\left(\avg{f}_Q^\sigma\right)^q.
\end{align*}


Given the disjointness of the sets $Q\in\md_\lambda$, \eqref{weakest1} is immediate
for $p=1$.  For $p>1$, notice the sequence 
$\{\sigma(Q)^{q/p}\}_{Q\in\md_\lambda}$ is $p,q$--Carleson with respect to the measure
$\sigma$.
\end{proof}

\section{Proof of Theorem \ref{maxbump}}
By Lemma \ref{mf2d}, Theorem \ref{maxbump} follows from the following lemma.
We remark that while the following proof does not make explicit use of the
Sawyer Maximal testing inequalities in \cite{Saw}, the proof does use 
some of the same ideas.
\begin{lm}\label{fracmo}
Let $1<p \leq q < \infty$, and let $\sigma$ and $w$ be two weights. 
Given a dyadic grid $\mathcal{D}$, let $M_{\alpha}$ be the dyadic fractional 
maximal operator. 
Let $\epsilon_q$ be a monotonic increasing function on $(1,\infty)$ that satisfies 
$\int_{1}^{\infty}\frac{dt}{t\epsilon_q^q(t)}=1$.  Define
\begin{align*}
\beta(Q)
:=\frac{\sigma(Q)^{1/p'}w(Q)^{1/q}}{\abs{Q}^{1-\alpha/n}}\rho_{\sigma}^{1/p}(Q)\epsilon_q(\rho_{\sigma}(Q)),
\end{align*}
Set $\mbump{\sigma}{w} := \sup_{Q\in\mathcal{Q}}\beta(Q)$, then
\begin{align*}
\norm{M_\alpha{f\sigma}}_{L^q(w)}
    \lesssim \mbump{\sigma}{w} \norm{f}_{L^p(\sigma)}.
\end{align*}
\end{lm}
\begin{proof}
\noindent Let $\m{S}$ be any sparse subset of $\md$. By Remark \ref{reduction} we need to verify 
\begin{align}\label{fracmote}
\int_{Q_0}\abs{\sum_{Q\in\mathcal{S}:Q\subset Q_0}\qan \avg{f\sigma}_Q\one_{E_Q}(x)}^{q}w(x)dx
\lesssim \mbump{\sigma}{w}^q\norm{f}_{L^p(\sigma)}^q.
\end{align} 

Let $\mathcal{Q}_k:=\{Q\in\mathcal{S}, Q\subset Q_0: \mbump{\sigma}{w} 2^{-k}\leq\beta(Q)\leq 
\mbump{\sigma}{w}2^{-k+1}\}$. 
We will show
\begin{align}\label{main}
\int_{Q_0}\abs{\sum_{Q\in\mathcal{Q}_k}\qan\avg{f\sigma}_Q\one_{E_Q}(x)}^qw(x)dx
\lesssim (2^{-k})^q\mbump{\sigma}{w}^q\norm{f}_{L^p(\sigma)}^q.
\end{align}
Taking $q^{th}$ roots and summing over $k$ will imply \eqref{fracmote}.

Using the identity $\avg{f\sigma}_Q=\avg{\sigma}_Q\avg{f}_{Q}^\sigma$ and the
pairwise disjointness of the sets $E_Q$, \eqref{main} will follow from: 
\begin{align*}
\sum_{Q\in\m{Q}_k}
\frac{\abs{Q}^{q\alpha/n}\sigma(Q)^qw(Q)}{\abs{Q}^q}(\avg{f}_{Q}^\sigma)^q
\lesssim (2^{-k})^q\mbump{\sigma}{w}^q\norm{f}_{L^p(\sigma)}^q.
\end{align*}
Thus, by the Carleson Embedding Theorem (Lemma \ref{cet}), it is enough to verify:
\begin{align*}
\frac{1}{\sigma(P)^{q/p}}\sum_{Q\in\m{Q}_k:Q\subset P}
  \frac{\abs{Q}^{q\alpha/n}\sigma(Q)^qw(Q)}{\abs{Q}^q}
\lesssim (2^{-k})^q\mbump{\sigma}{w}^q,
\end{align*}
for all $P\in\m{Q}_k$. Using the fact that 
$\beta(Q)\simeq 2^{-k}\mbump{\sigma}{w}$ for $Q\in\mathcal{Q}_k$ we estimate:
\begin{align*}
\sum_{Q\in\m{Q}_k:Q\subset P}
  \frac{\abs{Q}^{q\alpha/n}\sigma(Q)^qw(Q)}{\abs{Q}^q}
&=\sum_{Q\in\mathcal{Q}_k:Q\subset P}\frac{\abs{Q}^{q\alpha/n}
  \sigma(Q)^{q/p'}w(Q)}{\abs{Q}^{q}}\sigma(Q)^{q/p}
\\&\simeq (2^{-k})^q\mbump{\sigma}{w}^q \sum_{Q\in\mathcal{Q}_k:Q\subset P}
    \frac{1}{\rho_{\sigma}(Q)^{q/p}\epsilon_q^q(\rho_\sigma(Q))}\sigma(Q)^{q/p}.
\end{align*}

We want to show that the sum above is dominated by $\sigma(P)^{q/p}$.
To this end, set $\mathcal{S}_r=\{Q\in\mathcal{Q}_k, Q\subset P:2^{r-1}\leq \rho_\sigma(Q)
\leq 2^{r} \}$. Thus, the sum above is dominated by
\begin{align*}
\sum_{r=0}^{\infty}\frac{1}{2^{rq/p}\epsilon_q^q(2^r)}
\sum_{Q\in\mathcal{S}_r}\sigma(Q)^{q/p}.
\end{align*}
Appealing to the summability condition on $\epsilon_q$, it suffices to show that 
\begin{align}
\sum_{Q\in\mathcal{S}_r}\sigma(Q)^{q/p}
\leq 2^{qr/p}\sigma(P)^{q/p}. \label{fm2}
\end{align}
Let $\mathcal{S}_r^{\ast}$ be the maximal elements in $\mathcal{S}_r$. 
Observe that for fixed $S^\ast\in\mathcal{S}_r^\ast$, and for any
$P\subset S^{\ast}$, there holds:
\begin{align*}
\left(\int_{E_Q}\avg{\one_{S^*\sigma}}_Q \one_Q\right)^{q/p}
\leq \left(\int_{E_Q}\sup_{P\in\mathcal{D}}\avg{\one_{S^\ast}\sigma}_P \one_P\right)^{q/p}.
\end{align*}
Since the sets $E_Q$ are pairwise disjoint,  
$\abs{Q}\simeq \abs{E_Q}$, and $\int_{S^\ast}\sup_{P\in\mathcal{D}}\avg{\one_{S^\ast}\sigma}_P
\leq \sigma(S^\ast)\rho_\sigma(S^\ast)\simeq 2^r\sigma(S^\ast)$ for
$S^\ast\in\mathcal{S}_r^\ast$, we estimate
\begin{align*}
\sum_{Q\in\mathcal{S}_r}\sigma(Q)^{q/p}
&\leq\sum_{S^{\ast}\in\mathcal{S}_r^{\ast}}
    \sum_{Q\subset S^{\ast}}\left(\int_{E_Q}
    \sup_{P\in\mathcal{D}}\avg{\one_{S^\ast}\sigma}_{P}\one_P\right)^{q/p}
\\&\leq \sum_{S^{\ast}\in\mathcal{S}_r^{\ast}} 
    \left(\int_{S^*}\sup_{P\in\mathcal{D}}\avg{\one_{S^*}\sigma}_P\one_P\right)^{q/p}
\\&\lesssim 2^{qr/p} \sum_{S^{\ast}\in\mathcal{S}_r^{\ast}} \sigma^{q/p}(S^*).
\end{align*}
Using the disjointness of the sets $S^\ast\in\mathcal{S}_r^\ast$, the sum 
in the last line above is dominated by $\sigma(P)^{q/p}$, completing the
proof.

\end{proof}

\section{Proof of Theorem \ref{frac1bump}}
By Lemma \ref{ff2d}, Theorem \ref{frac1bump} follows from the following lemma.
\begin{lm}
Let $1<p\leq q < \infty$, and let $\sigma$ and $w$ be two weights. Given a dyadic grid $\mathcal{D}$,
let $I_\alpha^{\m{D}}$ be the dyadic fractional integral operator. 
Let $\epsilon_p$ be a monotone increasing function on $(1,\infty)$ such that $\int_{1}^{\infty}\frac{dt}{t\epsilon_p^p(t)}=1$,
and similarly for $\epsilon_{q'}$. Define
\begin{align*}
\beta(Q)
:=\frac{\sigma(Q)^{1/p'}w(Q)^{1/q}\qan}{\abs{Q}}
    \rho_\sigma(Q)^{1/p}\epsilon_p(\rho_{\sigma}(Q))
    \rho_{w}(Q)^{1/q'}\epsilon_{q'}(\rho_w(Q)).
\end{align*}
Set $\cbump{\sigma}{w} := \sup_{Q\in\mathcal{Q}}\beta(Q)$, then
\begin{align*}
\norm{I_\alpha^{\m{D}} (f\sigma)}_{L^q(w)}
    \lesssim \cbump{\sigma}{w} \norm{f}_{L^p(\sigma)}.
\end{align*}
\end{lm}
\begin{proof}
We proceed by duality. Let $f\in L^p(\sigma)$ and $g\in L^{q'}(w)$.
Below we use the identity: $\avg{f\sigma}_Q=\avg{f}_{Q}^{\sigma}\avg{\sigma}_Q$,
where $\avg{f}_Q^\sigma:=\sigma(Q)^{-1}\int_Q f\sigma$. Using the definition
of $\cbump{\sigma}{w}$, there holds:
\begin{align*}
\ip{\sum_{Q\in\mathcal{Q}}\qan\avg{f\sigma}_Q\one_{Q}}{g w}_{L^2(dx)}
&=\sum_{Q\in\mathcal{Q}}\avg{f}_{Q}^{\sigma}\avg{g}_Q^{w}
    \abs{Q}^{\alpha/n}\avg{\sigma}_{Q}w(Q)\qan
\\&=\sum_{Q\in\mathcal{Q}}\avg{f}_{Q}^{\sigma}\sigma(Q)^{\frac{1}{p}}\avg{g}_{Q}^{w}w(Q)^{\frac{1}{q'}}
    \frac{\sigma(Q)^{\frac{1}{p'}}w(Q)^{\frac{1}{q}}\qan}
    {\abs{Q}^{1-\frac{\alpha}{n}}}
\\&\lesssim \cbump{\sigma}{w} \sum_{Q\in\mathcal{Q}}
    \frac{\avg{f}_{Q}^{\sigma}\sigma(Q)^{\frac{1}{p}}}
    {\rho_{\sigma}^{\frac{1}{p}}(Q)\epsilon_{p}(\rho_{\sigma}(Q))}
    \frac{\avg{g}_{Q}^{w}w(Q)^{\frac{1}{q'}}}
    {\rho_{w}^{\frac{1}{q'}}(Q)\epsilon_{q'}(\rho_w(Q))}.
\end{align*}
By H\"older's inequality, it suffices to show that 
\begin{align*}
\left(
    \sum_{Q\in\mathcal{S}}
    \frac{\sigma(Q)}
    {\rho_{\sigma}(Q)\epsilon_p^p(\rho_{\sigma}(Q))}(\avg{f}_Q^{\sigma})^p\right)^{1/p} 
\hspace{.25in} \text{and}\hspace{.25in}
 \left(
    \sum_{Q\in\mathcal{S}}
    \frac{w(Q)^{p'/q'}}
    {\rho_{w}^{p'/q'}(Q)\epsilon_{q'}^{p'}(\rho_w(Q))}(\avg{g}_{Q}^{w})^{q'}
    \right)^{1/p'}
\end{align*}
are dominated by $\norm{f}_{L^p(\sigma)}$ and $\norm{g}_{L^{q'}(w)}$,
respectively. 
Since $p\leq q$, it follows that $q'\leq p'$, so by the the Carleson Embedding Theorem
(Lemma \ref{cet}), it suffices to show:
\begin{align*}
\sum_{Q\in\mathcal{S}:Q\subset P}\frac{\sigma(Q)}
    {\rho_{\sigma}(Q)\epsilon_p^p(\rho_\sigma(Q))} \lesssim\sigma(Q_0)
\hspace{.25in} \text{ and } \hspace{.25in}
\sum_{Q\in\mathcal{S}:S\subset P} \frac{w(Q)^{p'/q'}}
    {\rho_{w}^{p'/q'}(Q)\epsilon_{q'}^{p'}(\rho_w(Q))}\avg{g}_{Q}^{wq'} \lesssim w^{p'/q'}(Q_0)
\end{align*}
for all $Q_0\in\mathcal{S}$. But the proof of each of these
estimates is similar to those in Lemma \ref{fracmo} and we omit the details
\end{proof}

\section{Proof of Theorem \ref{sbump}}

From Remark \ref{reduction} and Lemma \ref{testingdyad}, it is enough to show 

\begin{align*} 
\int_{Q_0} \abs{\sum_{Q\in\m{Q}:Q\subset Q_0}\abs{Q}^{\alpha/n}\avg{\sigma}_Q\one_Q(x)}^{q}w(x)dx \lesssim \sbump{\sigma}{w}_{\alpha,p,q}\sigma(Q_0)^{q/p}
\end{align*}

\noindent for any sparse collection $\mQ$ and $Q_0\in \mQ$ (the dual testing condition follows identically).  For the remainder, fix a root $Q_0$ and let $\mQ$ be a sparse collection of cubes contained in $Q_0$.  Fix $\alpha,p,q$ in the respective appropriate range; we'll ignore these fixed indices where there is no confusion.  It remains to show

\begin{align*} 
\left\|\sum_{Q\in\m{Q}}\abs{Q}^{\alpha/n}\avg{\sigma}_Q\one_Q\right\|_{L^q(w,Q_0)}\lesssim \sbump{\sigma}{w}^{1/q}\sigma(P)^{1/p}.\label{stest}
\end{align*}

For $Q\in\mQ$, define \[\beta(Q):= \left(|Q|^{\alpha/n}\avg{\sigma}_Q\right)^{q/p'} \avg{w}_Q \varrho_\sigma(Q)\eps_q\left(\varrho_\sigma(Q)\right).\]  For integers $a$ and $r$, set $\mQ^{a,r}:=\{Q\in\mQ : \beta(Q) \simeq 2^a, \varrho(Q) \simeq 2^r\}$; notice $\mQ^{a,r}$ is empty for $a$ large enough.  Construct a stopping family $\mS$ for the $\sigma$ fractional averages: let $\mS$ be the minimal subset of $\mQ^{a,r}$ containing the maximal cubes in $\mQ^{a,r}$ such that whenever $S\in\mS$, the maximal cubes $Q\subset S$, $Q\in\mQ^{a,r}$ with $|Q|^{\alpha/n}\avg{\sigma}_Q>4|S|^{\alpha/n}\avg{\sigma}_S$ are also in $\mS$.  Denote by $Q^\mS$ the $\mS$--parent of $Q$.  Partition $\mQ^{a,r}$ into $\mQ^{a,r}_k$, those cubes in $\mQ^{a,r}$ such that $|Q|^{\alpha/n}\avg{\sigma}_Q \simeq 2^{-k}|Q^\mS|^{\alpha/n} \avg{\sigma}_{Q^\mS}.$  We temporarily denote $\mQ^{a,r}_k$ by $\mQ'$.  We will show

\begin{align}
\left\|\sum_{Q\in\mQ'} |Q|^{\alpha/n}\avg{\sigma}_Q \one_Q\right\|_{L^q(w)}\lesssim 2^{-k} \left[\sum_{S\in\mS} |S|^{q\alpha/n}\avg{\sigma}_S^q w(S)\right]^{1/q},\label{kpholed}
\end{align}

\noindent where summing over $k\geq -2$ gives 

\begin{align} 
\norm{\sum_{Q\in\mQ^{a,r}} |Q|^{\alpha/n}\avg{\sigma}_Q\one_Q}_{L^q(w)}\lesssim \left[\sum_{S\in\mS} |S|^{q\alpha/n}\avg{\sigma}_S^q w(S)\right]^{1/q} .\label{arpholed}
\end{align}


Define for each $S\in\mS$

\begin{align*}
\Phi_S:=\sum_{Q\in\mQ':Q^\mS =S} |Q|^{\alpha/n} \avg{\sigma}\one_Q \,\,\,\,\,\,\,\,\,\, &\text{and} \,\,\,\,\,\,\,\,\,\, \Phi_{S,\ell}:=\Phi_S  \one_{\{\Phi_S \simeq \ell 2^{-k}|S|^{\alpha/n}\avg{\sigma}_S\}}.
\end{align*}

\noindent Since $\sum\limits_{S\in\mS} \Phi_{S,\ell}$ is geometric for fixed $\ell\in\Z^+,$ H\"older's inequality yields

\begin{align}
\left(\sum_{\ell\geq 1} \sum_{S\in\mS} \Phi_{S,\ell}\right)^q \lesssim \sum_{\ell\geq 1} \ell^{2q/q'} \left(\sum_{S\in\mS} \Phi_{S,\ell} \right)^q \simeq \sum_{\ell\geq 1} \ell^{2q/q'} \sum_{S\in\mS} \Phi_{S,\ell}^q.\label{elltrick}
\end{align}

\noindent It is apparent that we need the following distributional estimate.

\begin{lm}\label{distest} There holds \[w\left\{\Phi_{S} >\lambda 2^{-k} 
|S|^{\alpha/n}\avg{\sigma}_S\right\} \lesssim 2^{- \lambda}w(S).\]
\end{lm}

\begin{proof}
The inequality is immediate in the case $w$ is Lebesgue measure from sparseness of $\mQ$.  Notice that we have for $Q\in \mQ'$ with $Q^\mS =S$, \[\avg{w}_Q \simeq \frac{2^a}{2^r \eps_q(2^r)}(2^{-k}\avg{\sigma}_{S}|S|^{\alpha/n})^{-q/p'}=:\tau_{S},\] where the equivalence is independent of $S$.  Denote by $\mQ^*$ the maximal cubes in $\mQ'$.  Since the $\left\{\Phi_{S} >\lambda 2^{-k} 
|S|^{\alpha/n}\avg{\sigma}_S\right\}$ is the union of the maximal cubes $P\in\mQ'$ with $P^\mS=S$ and $ \inf\limits_{x\in P}\Phi_S(x)>\lambda 2^{-k}|S|^{\alpha/n}\avg{\sigma}_S$, hence a disjoint union, it follows that 

\begin{align*}
w\left\{\Phi_{S} >\lambda 2^{-k} |S|^{\alpha/n}\avg{\sigma}_S\right\} &\simeq \tau_S \left| \left\{\Phi_S>\lambda 2^{-k} |S|^{\alpha/n} \avg{\sigma}_S\right\}\right|\\
&\lesssim \tau_S \left(2^{-(\lambda-1)}\sum_{Q^*\in\mQ^*}|Q^*|\right)\\
&\simeq 2^{-\lambda}\sum_{Q^*\in\mQ^*}w(Q^*).
\end{align*}

The collection $\mQ^*$ is disjoint, so the proof is complete.

\end{proof}

Since $\{\Phi_{S,\ell}>\lambda 2^{-k}|S|^{\alpha/n}\avg{\sigma}_S\}$ is constant for $0<\lambda<\frac{\ell}{2}$ and is empty for $\lambda>\ell$, we have 

\begin{align*}
\int_{Q_0} \Phi_{S,\ell}^q dw &= 2^{-kq}|S|^{q \alpha/n}\avg{\sigma}_S^q \int_0^\infty q\lambda^{q-1} w\{\Phi_{S,\ell}>\lambda 2^{-k}|S|^{\alpha/n}\avg{\sigma}_S\}d\lambda\\
&\lesssim 2^{-kq}|S|^{q \alpha/n}\avg{\sigma}_S^q \left[ \left(\frac{\ell}{2}\right)^q 2^{-\ell/2}w(S)+ \frac{\ell}{2} q\ell^{q-1} 2^{-\ell/2}w(S) \right] \\
&\simeq 2^{-kq}|S|^{q \alpha/n}\avg{\sigma}_S^q \left[ \ell^q 2^{-\ell/2} w(S) \right], \\
\end{align*}

\noindent where the second inequality is the application of Lemma \ref{distest}.  Recalling \eqref{elltrick}, this gives \eqref{kpholed}. 

For each $S$ define $E_S$ to be $S$ less the members of $\mS$ properly contained in $S$.  Let $\mS^*$ be the maximal elements of $\mS$.  Since $\beta(S)\simeq 2^a$ and $\varrho(S)\simeq 2^r$ for all $S\in\mS$, the right hand side of \eqref{arpholed} is equivalent to

\begin{align*}
\left(\frac{2^a}{2^{r}\eps_q(2^r)} \sum_{S\in\mS} \left(|S|^{\alpha/n}\avg{\sigma}_S\right)^{q/p} |S|\right)^{1/q} &\lesssim \left[  \frac{2^a}{2^{r}\eps_q(2^r)} \left( \sum_{S^*\in\mS^*} \sum_{S^*\supseteq S\in\mS} \int_{E_S} M_\alpha(\one_{S*}\sigma)^\frac{q}{p} dx\right) \right]^{1/q}\\
&\simeq \left[  \frac{2^a}{\eps_q(2^r)} \left( \sum_{S^*\in\mS^*}  \sigma(S^*)^{q/p}\right) \right]^{1/q}\\
    &\lesssim (2^{1/q})^a \frac{1}{\eps^{1/q}_q(2^r)} \sigma(Q_0)^{1/p}.\\
\end{align*}

The first inequality above follows from $|S|\simeq |E_S|=\int_{E_S} dx$, and the third from $p\leq q$.  Summing over integers $r\geq 0$ evokes the integrability condition on $\eps_q$; summing over relevant integers $a$ completes the proof.

\section{Acknowledgements} 
We thank Michael Lacey for suggesting this topic and for useful conversations and suggestions.
We also thank Brett Wick for helpful comments regarding the presentation of the paper. 
We would also like to thank Kabe Moen for alerting us to a mistake that
we made in Theorem 1.1. 



\end{document}